\documentclass[preprint,11pt, 2014]{elsarticle}

\usepackage[all]{xy}
\usepackage{amsmath,amssymb,amsthm}
\usepackage{graphicx,epsfig}
\usepackage{enumerate}
\usepackage{tikz}
\usepackage{ytableau}
\usepackage{hyperref}

\numberwithin{equation}{section}

\swapnumbers
\theoremstyle{plain}
\newtheorem{theorem}{Theorem}[section]

\newtheorem{corollary}[theorem]{Corollary}

\newtheorem{proposition}[theorem]{Proposition}

\theoremstyle{definition}

\newtheorem{remark}[theorem]{Remark}
\newtheorem{definition}[theorem]{Definition}

 \definecolor{light-gray}{gray}{0.96}

\journal{Discrete Mathematics}

\title{On the enumeration of $k$-omino towers}
\author{Tricia Muldoon Brown\\
Armstrong State University\\
11935 Abercorn Street, Savannah, GA 31419}

\begin{document}
\begin{abstract}
We describe a class of fixed polyominoes called $k$-omino towers that are created by stacking rectangular blocks of size $k\times 1$ on a convex base composed of these same $k$-omino blocks.  
By applying a partition to the set of $k$-omino towers of fixed area $kn$, we give a recurrence on the $k$-omino towers therefore showing the set of $k$-omino towers is enumerated by a Gauss hypergeometric function.  The proof in this case implies a more general hypergeometric identity with parameters similar to those given in a classical result of Kummer.
\end{abstract}

\maketitle

\section{Introduction}
If you ever played with LEGO\textsuperscript{\textregistered}s or DUPLO\textsuperscript{\textregistered}s as a child, you probably tried to stack the blocks into large towers, sometimes in a very haphazard way.  An interesting enumerative question immediately arises: How many different towers can be created?  In this manuscript, we fix the orientation and size of the blocks, called {\it $k$-ominoes}, to have unit width and length $k$ units.   We ask, for a given $k$, how many towers may be created using $k$-ominoes?  It turns out, that these questions are related to a larger enumerative question with a rich history.

A \textit{polyomino} is a collection of unit squares with incident sides.  Solomon Golomb is credited with bringing the attention of the mathematical community to polyominoes through an article published in the American Mathematical Monthly in 1954~\cite{Golomb_AMM} and his book on polyominoes~\cite{Golomb} published first in 1965.  Researchers from various disciplines have been interested in polyominoes, from chemists and physicists, to statisticians and recreational mathematicians.  Many polyomino problems lie in the area of mathematical problems that are easy to describe, but often surprisingly difficult to solve.  For example, no formula is known for the number of polyominoes parametrized by area, although asymptotic bounds do exist; see Klarner and Rivest~\cite{Klarner_Rivest_2} for the upper bound and Barequet, Rote, and Shalah~\cite{Barequet_Rote_Shalah} for a recently improved lower bound.  This general case may be unknown, but much work has been done to enumerate classes of polyominoes using parameters such as area, perimeter, length of base or top, number of rows or columns, and others.  

One common technique employed in these types of problems is to associate classes of polyominoes with words in algebraic languages, for example, the classical association between parallelogram polyominoes of perimeter $(2n+2)$ and Dyck words of length $2n$.  Continuing in this tradition are results from  Barcucci, Frosini and Rinaldi~\cite{Barcucci_Frosini_Rinaldi} who use the ECO method to associate polyominoes bounded by rectangles with Grand-Dyck and Grand-Motzkin paths, or the work by Delest and Viennot~\cite{Delest_Viennot}, Domo\c{c}os~\cite{Domocos}, or Castiglione et.\ al.~\cite{Castiglione_etal_2007} who associate classes of polyominoes with words in regular languages.  Another widely used technique is to dissect or decompose the polyomino into smaller parts.  This technique is employed by Wright~\cite{Wright}, Klarner and Rivest~\cite{Klarner_Rivest}, and Bender~\cite{Bender}. More recently, decompositions are used by Duchi, Rinaldi and Schaeffer~\cite{Duchi_Rinaldi_Schaeffer} in the enumeration by inflation of Z-convex polyominoes, and by Fedou, Frosini, and Rinaldi~\cite{Fedou_Frosini_Rinaldi} who utilize decomposition and inclusion-exclusion to count 4-stack polyominoes.  In this paper, we will employ a recursion on the base of a polyomino tower.  Recursions on classes of polyominoes have been used by Barcucci et. al~\cite{Barcucci_Lungo_Fedou_Pinzani} whose recursion is dependent on the rightmost column of a steep staircase polyomino, and Castiglione et.\ al.~\cite{Castiglione_etal_2005} who also use the ECO method to give a recurrence relation on the number of L-convex polyominoes with a given semi-perimeter.  Recent work by Bouvel, Guerrini, and Rinaldi~\cite{Bouvel_Guerrini_Rinaldi} using succession rules and Boussicault, Rinaldi, and Socci~\cite{Boussicault_Rinaldi_Socci} applying a bijection with ordered triples of certain trees and a lattice path provide further examples of strategies to enumerate classes of polyominoes.

Here, we only consider \textit{fixed polyominoes}, also called a fixed animals, which are oriented polyominoes such that different orientations of the same free shape are considered distinct.  Specifically, we study subsets of fixed polyominoes that have been derived from stacked towers of dominoes or $k$-ominoes.  Section~\ref{domino_towers} defines and enumerates the class of polyominoes inspired by domino building blocks, which we call \textit{domino towers}, and Section~\ref{k-omino_towers} generalizes these results to other horizontal $k$-omino blocks in terms of hypergeometric functions with the main result as follows:
\begin{theorem}\label{count_komino_theorem}
The number of $k$-omino towers with area $nk$, $D_k(n)$, is given by
\begin{equation*}
D_k(n)= {kn-1\choose n-1} {}_2F_1 (1, 1-n; (k-1)n+1; -1)
\end{equation*}
where ${}_2F_1(a,b;c;z)$ is the Gauss hypergeometric function and $k, n\geq 1$.
\end{theorem} 

\section{Domino towers}\label{domino_towers}
In order to provide insight on the general case, we begin with the case where $k=2$.  A {\it domino} is a $2$-omino block which is two units in length and has two \textit{ends}, a \textit{left end} and a \textit{right end}.  Similarly, the boundary of the vertical face of a collection of incident domino blocks may also define a fixed polyomino.  In such a collection, a domino is in the \textit{base} if no dominoes are or could be underneath it, and the \textit{level} of a domino will be the vertical distance from the base. Domino towers in terms of their area $2n$ and base of length $2b$ are defined as follows:

\begin{definition} For $n\geq b\geq 1$, an \textit{(n,b)-domino tower} is a fixed polyomino created by sequentially stacking $n-b$ dominoes horizontally on a convex, horizontal base composed of $b$ dominoes, such that if a non-base domino is placed in position indexed by $\{(x,y),(x+1,y)\}$, then there must be a domino in position $\{(x-1,y-1), (x,y-1)\}$, $\{(x,y-1), (x+1,y-1)\}$, or $\{(x+1,y-1), (x+2,y-1)\}$.
\end{definition}

We note, that all dominoes are placed with the same horizontal orientation in space so that the dimensions are two along the horizontal axis and one along the other axis.  Next, we define a \textit{column} of a polyomino or of the corresponding domino tower as the intersection of the polyomino with an infinite vertical line of unit squares.  Further, within a domino tower, a domino $\mathbf{x}$ is \textit{supported} by another domino $\mathbf{y}$ if there is a chain of dominoes $\mathbf{y} = \mathbf{x_0}, \mathbf{x_1}, \ldots, \mathbf{x_m} = \mathbf{x}$ such that the level of $\mathbf{x_i}$ is one less than the level of $\mathbf{x_{i+1}}$ and the intersection of the column set of~$\mathbf{x_i}$ and the column set of~$\mathbf{x_{i+1}}$ is non-empty for all $0\leq i \leq m-1$.  We say a domino~$\mathbf{x}$ is \textit{completely supported} by another domino~$\mathbf{y}$ if the removal of $\mathbf{y}$ from the tower would cause $\mathbf{x}$ to drop at least one level, that is, the domino~$\mathbf{x}$ could not have been stacked in the tower without first placing domino $\mathbf{y}$. 

Before we state the first theorem, we give another interpretation of domino towers in order to describe two constructions on these towers.  Viennot~\cite{Viennot} introduced a class of objects called \textit{heaps} which, as a set, contain the set of domino towers.  We adapt one of Viennot's equivalent definitions of a heap to represent a domino tower as an ordered triple $(E, \leq, \epsilon)$ where $E$ is a finite poset with order relation~$\leq$ and $\epsilon$ is a map from the poset $E$ to the set of basic positions, that is, a map designating a block's horizontal location or column set.  This triple $(E,\leq, \epsilon)$ satisfies two conditions:
\begin{enumerate}
\item Given dominoes $\mathbf{x}$ and $\mathbf{y}$, if $\epsilon(\mathbf{x})$ is {\it concurrent} with $\epsilon(\mathbf{y})$, that is, the column set of $\mathbf{x}$ intersects the column set of $\mathbf{y}$, then either $\mathbf{x} \leq \mathbf{y}$ or $\mathbf{y} \leq \mathbf{x}$.
\item For every $\mathbf{x}$ and $\mathbf{y}$, the relation $\mathbf{x} \leq \mathbf{y}$ implies $\mathbf{y}$ is supported $\mathbf{x}$.
\end{enumerate}   
 
Now, we say a domino $\mathbf{y}$ is \textit{deleted} from a domino tower when $\mathbf{y}$ has been removed from all chains containing it in the poset generated by $\leq$.  We note, order relations among other dominoes in the poset must be redrawn as $\mathbf{x} \leq \mathbf{y} \leq \mathbf{z}$ does not necessarily imply $\mathbf{x}$ supports $\mathbf{z}$ after $\mathbf{y}$ has been deleted.  Further, we {\it grow} a domino tower at position given by $\epsilon(\mathbf{x})$ and level $\ell$, if for all dominoes $\mathbf{y}$ at level $\ell$ such that $\epsilon(\mathbf{x})$ is concurrent with $\epsilon(\mathbf{y})$, we replace the domino $\mathbf{y}$ with the relation $\mathbf{x} \leq \mathbf{y}$ in all chains containing $\mathbf{y}$ in the poset given by $\leq$, that is, we insert $\mathbf{x}$ to be covered by $\mathbf{y}$ in these chains.  In both these constructions, only the poset $E$ is altered in the formation of a new domino tower; neither the order relation $\leq$ nor the map $\epsilon$ change.  Essentially, deleting a domino causes all dominoes that are completely supported by that domino to fall vertically, with no change in their horizontal position, until they come to rest on another domino or the base, and growing a tower at a certain position in the tower causes all dominoes supported by that position to be lifted vertically one level, also with no change in their horizontal position.

We note that in some cases these constructions are inverses; if the dominoes in the deletion of $\mathbf{y}$ fall at most one level, then growing by $\mathbf{y}$ moves the dominoes back up to their original position.  Further, deleting the growth in any domino tower gives the original tower as all the dominoes affected by the growth of a domino $\mathbf{y}$ are consequently completely supported by $\mathbf{y}$.

Now, using parameters area and base, domino towers may be counted with binomial coefficients. 

\begin{theorem}\label{nb_domino_tower}
The number of $(n,b)$-domino towers, $d_b(n)$, is given by $2n-1\choose n-b$ for $n\geq b\geq 1$.
\end{theorem}

\begin{proof}
We proceed by induction on $n$ and $b$.  Clearly, the number of $(1,1)$-domino towers is given by ${1\choose 0}=1$, so assume $n>1$.  
Applying a case of the Vandermonde identity, for $n\geq 2$ we have
\begin{equation}\label{domino_case_equation}
{2n-1\choose n-b} = {2(n-1)-1\choose (n-1)-(b-1)}+2{2(n-1)-1\choose (n-1)-b} + {2(n-1)-1\choose (n-1)-(b+1)}.
\end{equation}
In the case where the base of a $n$-domino tower is a single domino, that is $b=1$, the tower could have been created by placing a $(n-1,1)$-domino tower on one of three positions on the base domino (left, right, or middle), or by centering a $(n-1,2)$-domino tower on the single domino base.  Hence the number of $(n,1)$-domino towers is given by three times the number $(n-1,1)$-domino towers plus the number of $(n-1,2)$-domino towers and 
\begin{equation*}
3{2(n-1)-1\choose n-2} +{2(n-1)-1\choose n-3} = {2(n-1)-1\choose n-1} + 2{2(n-1)-1\choose n-2}+ {2(n-1)-1\choose n-3}
\end{equation*}
satisfies the recurrence.

Now, assuming $n\geq b >1$, we will show that the $(n,b)$-domino towers may be built from domino towers of $(n-1)$ blocks and bases of length $b-1$, $b$, and $b+1$.  To begin, we partition the set of $(n,b)$-domino towers into four disjoint sets as follows:
\begin{enumerate}
\item Let $A_{n,b}$ be the set of $(n,b)$-domino towers such that the leftmost domino on the first level does not intersect the column containing the left end of the leftmost domino of the base. 
\item Let $B_{n,b}$ be the set of $(n,b)$-domino towers that have a domino on the first level whose left end intersects the column containing the left end of the leftmost domino in the base.
\item Let $C_{n,b}$ be the set of $(n,b)$-domino towers that have a domino placed on the first level so that its right end intersects the column containing the left end of the leftmost domino of the base and whose left end extends past the base on the left side.  Further, assume the rightmost domino on the first level of the tower does not extend past the base on the right side, that is, the column containing the right end of the rightmost domino on the first level must intersect the base.
\item Let $D_{n,b}$ be the set of $(n,b)$-domino towers that have a domino placed on the first level so that its right end intersects the column containing the left end of the leftmost domino of the base and whose left end extends past the base on the left side.  Further, the rightmost domino on the first level must extend one unit past the base on the right side, that is, the column containing the right end of the rightmost domino on the first level does not intersect the base.
\end{enumerate}
These sets are illustrated in Figure~\ref{disjointsets_fig} in the case of $n=4$ and $b=2$.  
\begin{figure}
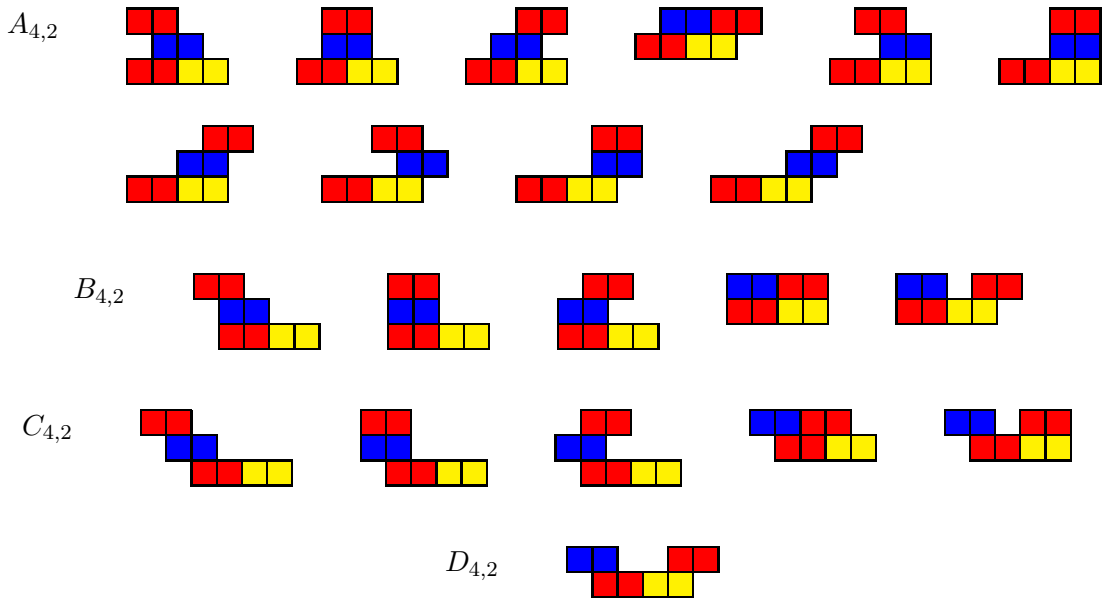

\centering
\ytableausetup{smalltableaux}
$A_{4,2}$
\qquad
\begin{ytableau}
*(red) & *(red)\\
\none & *(blue) &*(blue)\\
*(red) & *(red) &*(yellow) &*(yellow) 
\end{ytableau}
\qquad
\begin{ytableau}
\none&*(red) & *(red)\\
\none & *(blue) &*(blue)\\
*(red) & *(red) &*(yellow) &*(yellow) 
\end{ytableau}
\qquad
\begin{ytableau}
\none & \none &*(red) & *(red)\\
\none & *(blue) &*(blue)\\
*(red) & *(red) &*(yellow) &*(yellow) 
\end{ytableau}
\qquad
\begin{ytableau}
\none & *(blue) &*(blue) &*(red) &*(red)\\
*(red) & *(red) &*(yellow) &*(yellow) 
\end{ytableau}
\qquad
\begin{ytableau}
\none &*(red) & *(red)\\
\none &\none & *(blue) &*(blue)\\
*(red) & *(red) &*(yellow) &*(yellow) 
\end{ytableau}
\qquad
\begin{ytableau}
\none &\none &*(red) & *(red)\\
\none &\none & *(blue) &*(blue)\\
*(red) & *(red) &*(yellow) &*(yellow) 
\end{ytableau}
\vspace{.2in}
\newline
\qquad
\begin{ytableau}
\none&\none & \none &*(red) & *(red)\\
\none &\none & *(blue) &*(blue)\\
*(red) & *(red) &*(yellow) &*(yellow) 
\end{ytableau}
\qquad
\begin{ytableau}
\none&\none  &*(red) & *(red)\\
\none &\none &\none & *(blue) &*(blue)\\
*(red) & *(red) &*(yellow) &*(yellow) 
\end{ytableau}
\qquad
\begin{ytableau}
\none&\none & \none &*(red) & *(red)\\
\none &\none &\none & *(blue) &*(blue)\\
*(red) & *(red) &*(yellow) &*(yellow) 
\end{ytableau}
\qquad
\begin{ytableau}
\none&\none & \none &\none &*(red) & *(red)\\
\none &\none &\none & *(blue) &*(blue)\\
*(red) & *(red) &*(yellow) &*(yellow) 
\end{ytableau}
\vspace{.3in}
\newline


$B_{4,2}$
\qquad
\begin{ytableau}
*(red) & *(red)\\
\none & *(blue) &*(blue)\\
\none &*(red) & *(red) &*(yellow) &*(yellow) 
\end{ytableau}
\qquad
\begin{ytableau}
*(red) & *(red)\\
 *(blue) &*(blue)\\
*(red) & *(red) &*(yellow) &*(yellow) 
\end{ytableau}
\qquad
\begin{ytableau}
\none&*(red) & *(red)\\
*(blue) &*(blue)\\
*(red) & *(red) &*(yellow) &*(yellow) 
\end{ytableau}
\qquad
\begin{ytableau}
*(blue) &*(blue)&*(red) & *(red) \\
*(red) & *(red) &*(yellow) &*(yellow) 
\end{ytableau}
\qquad
\begin{ytableau}
*(blue) &*(blue) &\none &*(red) & *(red) \\
*(red) & *(red) &*(yellow) &*(yellow) 
\end{ytableau}
\qquad
\vspace{.3in}


$C_{4,2}$
\qquad
\begin{ytableau}
*(red) & *(red)\\
\none & *(blue) &*(blue)\\
\none & \none & *(red) & *(red) &*(yellow) &*(yellow) 
\end{ytableau}
\qquad
\begin{ytableau}
*(red) & *(red)\\
 *(blue) &*(blue)\\
 \none & *(red) & *(red) &*(yellow) &*(yellow) 
\end{ytableau}
\qquad
\begin{ytableau}
\none &*(red) & *(red)\\
 *(blue) &*(blue)\\
 \none & *(red) & *(red) &*(yellow) &*(yellow) 
\end{ytableau}
\qquad
\begin{ytableau}
 *(blue) &*(blue)& *(red) & *(red)\\
 \none & *(red) & *(red) &*(yellow) &*(yellow) 
\end{ytableau}
\qquad
\begin{ytableau}
 *(blue) &*(blue)&\none& *(red) & *(red)\\
 \none & *(red) & *(red) &*(yellow) &*(yellow) 
\end{ytableau}
\vspace{.3in}
\newline
$D_{4,2}$
\qquad
\begin{ytableau}
  *(blue)&*(blue) &\none & \none &*(red) & *(red) \\
\none & *(red) & *(red) &*(yellow) &*(yellow) 
\end{ytableau}
\caption{The sets $A_{4,2}$, $B_{4,2}$, $C_{4,2}$, and $D_{4,2}$.}
\label{disjointsets_fig}
\end{figure}

Beginning with a tower from the set $A_{n.b}$, we observe that none of its dominoes are completely supported by the leftmost domino in the base.  Thus, this leftmost base domino can be deleted without affecting the rest of the tower.  Since, the process can be reversed, that is, we can grow any $(n-1,b-1)$-domino tower into a tower from the set $A_{n,b}$ by inserting a domino on level zero directly to the left of the base, the set $A_{n,b}$ is in bijection with the set of $(n-1,b-1)$-domino towers, and thus has cardinality $2n-3\choose n-b$.  

Next, given a tower in the set $B_{n,b}$, we proceed by deleting the leftmost domino from the first level.  As the deleted domino rests directly above a domino on level zero, the dominoes fall at most one level. So, from any $(n-1,b)$-domino tower we can grow a domino tower in the set $B_{n,b}$ by inserting a domino on level one just above the leftmost domino of the base, reversing the process.  Thus, $B_{n,b}$ is enumerated by $2n-3\choose n-b-1$.  

Now, we wish to show that $C_{n,b}$ is also in bijection with $(n-1, b)$-domino towers by constructing a bijection with $B_{n,b}$.  To construct the bijection, first fix all non-base dominoes in a tower from the set $C_{n,b}$, and then shift the base one unit, or half of a domino, to the left.  For example, shifting the base in the $C_{4,2}$ row in Figure~\ref{disjointsets_fig} produces the row $B_{4,2}$ above it.  In the inverse map, the base of a domino towers from the set $B_{n,b}$ are shifted to the right by one unit.  The map is well-defined by construction of $B_{n,b}$ and ${C}_{n,b}$, hence $C_{n,b}$ has cardinality $2n-3\choose n-b-1$.  

Finally, all towers in $D_{n,b}$ have dominoes on the first level extending over the base on both sides.  Thus we can remove the right end of the leftmost domino on the first level and the left end of the rightmost domino of the first level.  The remaining ends of these dominoes will drop down one level to bracket the base.  Finally, to complete the process replace the base of $b$ dominoes bracketed by unit squares with a base of $(b+1)$-dominoes; see Figure~\ref{bijectionD_fig}.   This is equivalent to deleting the leftmost and rightmost dominoes on level one and placing two unit squares on level zero to bracket the base.  This process can also be reversed; we can grow an $(n-1,b+1)$-domino tower by inserting two new dominoes on level one directly above the leftmost and rightmost dominoes of the base.  Then, in the base the left end of the leftmost domino and the right end of the rightmost domino can be removed to form a new base of length $2b$.  This gives a bijection between the set of $(n-1, b+1)$-domino towers and $D_{n,b}$, and consequently, the cardinality of $D_{n,b}$ is $2n-3\choose n-b-2$.  Thus, the claim has been proven.
\begin{figure}
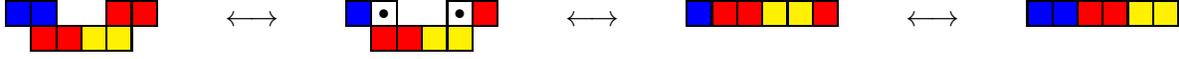

\centering
\ytableausetup{smalltableaux}
\begin{ytableau}
  *(blue)&*(blue) &\none & \none &*(red) & *(red) \\
\none & *(red) & *(red) &*(yellow) &*(yellow) 
\end{ytableau}
\qquad
$\longleftrightarrow$
\qquad
\begin{ytableau}
  *(blue)& \bullet &\none & \none & \bullet & *(red) \\
\none & *(red) & *(red) &*(yellow) &*(yellow) 
\end{ytableau}
\qquad
$\longleftrightarrow$
\qquad
\begin{ytableau}
  *(blue)& *(red) & *(red) &*(yellow) &*(yellow)  &*(red)
\end{ytableau}
\qquad
$\longleftrightarrow$
\qquad
\begin{ytableau}
  *(blue)& *(blue) & *(red) & *(red) &*(yellow) &*(yellow) 
\end{ytableau}
\caption{Illustration of the bijection between $D_{4,2}$ and the set of $(3,3)$-domino towers}
\label{bijectionD_fig}
\end{figure}
\end{proof}

As a consequence of Theorem~\ref{nb_domino_tower}, we state the following corollary.

\begin{corollary}\label{number_domino}
The number of $n$-domino towers is $4^{n-1}$ for $n\geq 1$.
\end{corollary}

This formula is found using the combinatorial identity 1.83 found in Gould~\cite{Gould} given below:
\begin{equation*}
\sum_{k=0}^n {2n+1\choose k} = 2^{2n}
\end{equation*}

We can also explicitly state the linear recurrence utilized in the proof.
\begin{corollary}\label{domino_recurrence}
The linear recurrence on the number of domino towers with $n$ dominoes and base of $b$ dominoes is
\begin{equation*}
d_b(n) = d_{b-1}(n-1) + 2 d_b (n-1) +d_{b+1} (n-1).
\end{equation*}
\end{corollary}

We apply the recursion to determine the bivariate generating function for the number of $(n,b)$-domino towers.

\begin{proposition}
The bivariate generating function $D_2(x,y)$ for the number of $(n,b)$-domino towers is
\begin{equation*}
D_2(x,y)=\sum_{n\geq 1} \sum_{b=1}^n d_b(n) x^n y^b = \sum_{n \geq 1} \sum_{b=1}^n {2n-1\choose n-b}x^n y^b = \frac{xy+(y-1)\frac{x}{2}\left(\frac{1-\sqrt{1-4x}}{\sqrt{1-4x}}\right)}{1-2x-xy-\frac{x}{y}}.
\end{equation*}
\end{proposition}

\begin{proof}
We apply the recurrence of Corollary~\ref{domino_recurrence}.
\small
\begin{eqnarray*}
D_2(x,y)&=&\sum_{n\geq 1} \sum_{b=1}^n d_b(n) x^n y^b \\
&=&\sum_{n\geq 1} {2(n-1)-1\choose (n-1)-(b-1)}x^ny^b+ 2{2(n-1)-1\choose (n-1)-b}x^ny^b+ {2(n-1)-1\choose (n-1)-(b+1)}x^ny^b\\
&=& xy+xyD_2(x,y) + y\sum_{n\geq 1} {2(n-1)-1\choose n-1}x^{n} + 2xD_2(x,y) +\frac{x}{y} D_2(x,y) - \sum_{n\geq 1}{2(n-1)-1\choose n-2}x^{n}\\
\end{eqnarray*}
\normalsize
Because
\begin{equation*}
\sum_{n\geq 1} {2(n-1)-1\choose n-1}x^n =\sum_{n\geq 1} {2(n-1)-1\choose n-2}x^n = \frac{x}{2}\left(\frac{1-\sqrt{1-4x}}{\sqrt{1-4x}}\right),
\end{equation*}
the result follows.
\end{proof}

\begin{remark}
There is another connection between domino towers and polyominoes through a natural bijection given by Viennot~\cite{Viennot} associating directed polyominoes with {\it strict pyramids} of dominoes, that is domino towers that have one domino in the base and the condition that no domino is placed directly above another domino.  The growth constant in this case is 3, as first proven by Dhar~\cite{Dhar}.  However a proof of this result using decompositions and generating functions is given by B\'etr\'ema and Penaud~\cite{Betrema_Penaud} and is further extended by  Bousquet-M\'elou and Rechnitzer~\cite{Bousquet_Rechnitzer} to enumerate larger classes of heaps of dominoes associated with classes of polyominoes whose growth constants are greater than known classes.   As domino towers are a subset of domino heaps one could use this bijection to describe a class of directed polyominoes whose growth constant in terms of area is $(4^{n-1})^{\frac{1}{2n}}= (1/4)^{\frac{1}{2n}}\cdot 2\sim 2$, this is, approximately half of that of the general fixed polyominoes which is estimated around 4.06~\cite{Jensen}.  (This result is obtained through a numerical analysis of the series of fixed polyominoes of limited size.)  However, we will not discuss this any further here as we are interested in the towers themselves as polyominoes.
\end{remark}

\section{$k$-omino towers}\label{k-omino_towers}

We generalize the results of Section~\ref{domino_towers} to horizontal polyominoes of integer length. 

\begin{definition} For $n\geq b\geq 1$, an \textit{(n,b)-$k$-omino tower} is a fixed polyomino created by sequentially placing $(n-b)$ horizontal $k$-ominoes on a convex, horizontal base composed of $b$ $k$-ominoes, such that if a non-base $k$-omino is placed in position $\{(x,y),(x+k-1,y)\}$, then there must be a $k$-omino in one of the $2k-1$ positions between $\{(x-k+1,y-1), (x,y-1)\}$ and $\{(x+k-1,y-1), (x+2k-2,y-1)\}$.
\end{definition}

As before, the orientation of the $k$-omino blocks is fixed.  Further we can easily generalize the notions of columns, support, deletion and growth to $k$-omino pieces.  The following result generalizes Theorem~\ref{nb_domino_tower}.

\begin{theorem}\label{theorem_knb}
The number of $(n,b)$-$k$-omino towers is given by $kn-1\choose n-b$ for $n, b\geq 1$ and $k\geq 1$.
\end{theorem}

\begin{proof}
Again, we utilize a case of the Vandermonde identity to obtain the equation
\begin{equation}\label{equation_k}
{kn-1\choose n-b} = \sum_{i=0}^{k} {k\choose i} {k(n-1)-1\choose (n-1)-(b+i-1)}.
\end{equation}
First, assume $b\geq 2$.  We will show $k$-omino towers can be built from towers of one less $k$-omino and bases of sizes from $b-1$ to $b+k-1$. 

To define the map, begin with a $(n,b)$-$k$-omino tower.  Let $L_j$ and $R_j$, respectively, represent the leftmost and rightmost, respectively, $k$-ominoes on level $j$ for integers $j\geq 0$. Identify $L_1$, the leftmost $k$-omino on level one of the tower, and $L_0$, the leftmost $k$-omino in the base.  Suppose the column containing the leftmost unit square of $L_0$ does not intersect $L_1$.  Then, analogously to the set $A_{n,b}$ described in the proof of Theorem~\ref{nb_domino_tower}, none of the dominoes in the $k$-omino tower are completely supported by $L_0$.  Thus we may remove $L_0$ to obtain a $(n-1,b-1)$-$k$-omino tower.  

Now assume $L_1$ is completely supported by $L_0$.  Let $1\leq k_0 \leq k$ be the number of columns which intersect both $L_1$ and $L_0$, and for $j\geq 1$, let $k_j$ be the number of columns which intersect both $R_{j}$ and $R_{j-1}$, provided the column through the rightmost square of $R_j$ does not intersect $R_{j-1}$.  Set $k_j=0$ otherwise, that is, if $R_j$ is stacked directly above or to the left of $R_{j-1}$ the value of $k_j$ is zero.  Identify the index $0\leq j\leq k-1$ such that 
 \begin{enumerate}
 \item $k_0 + k_1 + \cdots + k_j \leq k$ and
 \item $k_0 + k_1 + \cdots + k_j +k_{j+1} >k$ or $k_{j+1}=0$.
 \end{enumerate}
If  we have equality, that is, $k_0 + k_1 + \cdots + k_j=k$, our map is defined as follows:  Remove the $k$ unit squares of $L_1$, $R_1, \ldots, R_j$ enumerated by $k_0, k_1, \ldots, k_j$.  For each block $L_1$, $R_1, \ldots, R_j$, these are precisely the squares above the block's supporting $k$-omino one level below.  As the remaining $k(j+1)-k$ unit squares of these $k$-ominoes are unsupported, they fall to level zero leaving $(b+j)k$ unit squares which can be merged and to form a base of $(b+j)$ $k$-ominoes.  As in a deletion, any $k$-ominoes completely supported by $L_1$, $R_1, \ldots R_j$ also fall on this new base, leaving a $(n-1,b+j)$-$k$-omino tower.  Note, this map is an extension of the action on the set $D_{n,b}$ in the domino case.  Further, we observe if $k_0=k$, as in the set $B_{n,b}$ in Theorem~\ref{nb_domino_tower}, the entire block $L_1$ is removed and as $j=0$ in this case the length of the base remains the same.

Otherwise, if  $k_0 + k_1 + \cdots + k_j< k$, we will slide the base as we did with dominoes from the set $C_{n,b}$, but in this more general case we will also slide the right staircase consisting of all $k$-ominoes $R_j$ for $j\geq 1$.  Fix the $k$-ominoes in the base and in the right staircase in relation to one another and slide the blocks to the left $k-(k_0 + k_1 + \cdots + k_j)$ units.  In other words, the function $\epsilon$ on any block from the base or right staircase now takes the block to a new set of columns which are $k-(k_0 + k_1 + \cdots + k_j)$ units to the left of the original position.  For the remaining $k$-ominoes in the tower, $\epsilon$ remains fixed.  However, if $\mathbf{y}$ is a domino in the base or right staircase and $\mathbf{x}$ is a domino that is not in the base or right staircase, if, after the slide, $\mathbf{y}$ and $\mathbf{x}$ are concurrent we must have that $\mathbf{y} \leq \mathbf{x}$, that is, the base and the right staircase slides left and under the other $k$-ominoes.  After the slide, the intersection of $L_0$ and $L_1$ in the new tower measured by $k_0^*$ is such that  $k_0^* + k_1 + \cdots + k_j=k$, and therefore we can find a $(n-1,b+j)$-$k$-omino tower as described above.   See Figure~\ref{k-omino_bijection} for an illustration of this process.  The map is well-defined because if $k_{j+1}\not= 0$, the number of unit squares supporting the block $R_{j+1}$ must be greater than $k-(k_0+k_1+\cdots+k_j)$ in order to satisfy the two conditions above for identifying $k_0, k_1, \ldots, k_j$.

\begin{figure}
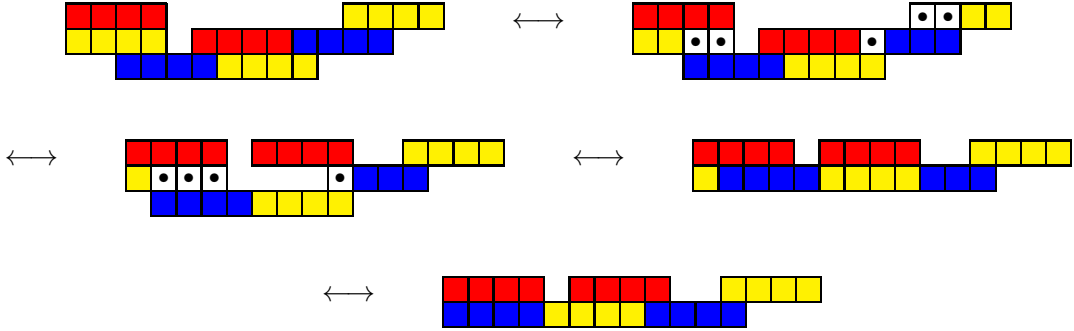

\centering
\ytableausetup{smalltableaux}
\begin{ytableau}
*(red)&*(red) &*(red) &*(red) & \none &\none & \none & \none & \none &\none & \none & *(yellow)& *(yellow) &*(yellow) &*(yellow) \\
*(yellow)&*(yellow) &  *(yellow) &*(yellow) & \none &*(red)&*(red) &*(red) &*(red)  & *(blue)& *(blue)&*(blue)&*(blue)\\
\none & \none & *(blue)& *(blue) & *(blue) & *(blue)& *(yellow)& *(yellow) &*(yellow) &*(yellow)\\
\end{ytableau}
\qquad
$\longleftrightarrow$
\qquad
\begin{ytableau}
*(red)&*(red) &*(red) &*(red) & \none &\none & \none & \none &\none & \none & \none & \bullet & \bullet &*(yellow) &*(yellow) \\
*(yellow)&*(yellow) & \bullet & \bullet & \none &*(red)&*(red) &*(red) &*(red)  &\bullet & *(blue)&*(blue)&*(blue)\\
\none & \none & *(blue)& *(blue) & *(blue) & *(blue)& *(yellow)& *(yellow) &*(yellow) &*(yellow)\\
\end{ytableau}
\qquad
\vspace{.3in}

$\longleftrightarrow$
\qquad
\begin{ytableau}
*(red)&*(red) &*(red) &*(red) & \none  & *(red)&*(red) &*(red) &*(red) &\none & \none & *(yellow)& *(yellow) &*(yellow) &*(yellow)\\
*(yellow)& \bullet& \bullet& \bullet  &\none &\none &\none &\none & \bullet & *(blue)&*(blue)&*(blue)\\
 \none & *(blue)& *(blue) & *(blue) & *(blue)& *(yellow)& *(yellow) &*(yellow) &*(yellow)\\
\end{ytableau}
\qquad
$\longleftrightarrow$
\qquad
\begin{ytableau}
*(red)&*(red) &*(red) &*(red)  & \none &*(red)&*(red) &*(red) &*(red) &\none &  \none & *(yellow)& *(yellow) &*(yellow) &*(yellow)\\
*(yellow)& *(blue)& *(blue) & *(blue) & *(blue)& *(yellow)& *(yellow) &*(yellow) &*(yellow)& *(blue)&*(blue)&*(blue)\\
\end{ytableau}
\qquad
\vspace{.3in}

$\longleftrightarrow$
\qquad
\begin{ytableau}
*(red)&*(red) &*(red) &*(red)  & \none &*(red)&*(red) &*(red) &*(red) &\none &  \none & *(yellow)& *(yellow) &*(yellow) &*(yellow)\\
 *(blue)& *(blue) & *(blue) & *(blue)& *(yellow)& *(yellow) &*(yellow) &*(yellow)& *(blue)&*(blue)&*(blue)&*(blue)\\
\end{ytableau}

\caption{A $(7,2)$-$4$-omino tower, where $k_0=2$, $k_1=1$, and $k_2=2$, mapped to a $(6,3)$-$4$-omino tower.}
\label{k-omino_bijection}
\end{figure}  

Further, given a fixed $(n-1,b+j)$-$k$-omino tower, $T$, the $(n,b)$-$k$-omino towers which map onto $T$ are those with compositions $k_0+k_1+\cdots+k_j=k$ along with those who have slid $k_0-1$ ways for each composition.  Thus the number of such $(n,b)$-$k$-omino towers is given by  
\begin{equation*}
\sum_{k_0+k_1+\cdots+k_j=k} k_0 = {k\choose j+1}, 
\end{equation*}
where the sum is over compositions $k_0+ k_1+\cdots+k_j$ of $k$ into non-zero $j+1$ parts.  The equality follows by a simple inductive argument where
\begin{eqnarray*}
{k\choose j+1} &=& {k-1\choose j+1}+{k-1\choose j} = \sum_{k_0+k_1+\cdots+k_j=k-1} k_0+ \sum_{k_0+k_1+\cdots+k_{j-1}=k-1} k_0\\
& = &\sum_{k_0+k_1+\cdots+(k_j+1)=k} k_0 + \sum_{k_0+k_1+\cdots+k_{j-1}+1=k} k_0 = \sum_{k_0+k_1+\cdots+k_j=k} k_0.
\end{eqnarray*}
because compositions of $k$ into $j+1$ non-zero parts can be partitioned into compositions whose $(j+1)$st part is one and compositions whose $(j+1)$st part is greater than one.
Therefore the set of $(n,b)$-$k$-omino towers can be partitioned into sets which are indexed by $0\leq j\leq k-1$ and determined by the compositions $k_0+k_1+\cdots+k_j$.  Each of these sets maps onto $k\choose j+1$ copies of the set of $(n-1,b+j)$-$k$-omino towers, or equivalently, $k\choose i$ copies of the set of $(n-1, b+i-1)$-$k$-omino towers for $1\leq i \leq k$.

We can illustrate how this map partitions $(n,b)$-domino towers into sets identified by the composition $k_0+k_1+\cdots +k_j=k$ or the inequality $k_0+k_1+\cdots +k_j<k$ as applied to the domino example.  In the proof of Theorem~\ref{nb_domino_tower}, the set $B_{n,b}$ is analogous to the set given by $k_0=2$ with no shift needed.  In this case, when $k_0=k$, the entire block $L_1$ is deleted leaving a domino tower with one less domino and the same length base.  Additionally, the set $C_{n,b}$ is described by $k_0=1<2$ and $k_1=0$ and thus is shifted by one unit.  Finally, the set $D_{n,b}$ has the property $k_0=1$ and $k_1=1$, thus $1+1=2$ and no shift is used.

To check uniqueness, consider the inverse map which takes a $(n-1,b+j)$-$k$-omino tower, $T$, onto $k\choose j+1$ $(n,b)$-$k$-omino towers using compositions $k_0+k_1+\cdots +k_j=k$ and slides from zero to $k_0-1$ squares. To apply a slide of one unit, we slide the base of the $(n,b)$-$k$-omino tower and the blocks $R_1, \ldots, R_j$ one unit to the right, and we see that the value of $k_0$ in the new tower decreases by one, which allows for a corresponding composition $(k_0-1)+k_1+\cdots+k_j+1=k$ with the same shape and $j+2$ parts.  However, this implies the tower $T$ contains the $k$-omino $R_{j+1}$ whose support by $R_j$ after the slide is one unit, and hence before the slide it would not have been supported by the tower.  This contradicts the fact that the base is $(b+j)$ $k$-ominoes, and therefore the compositions $(k_0-1)+k_1+\cdots+k_j+1=k$ must be associated to a $(n-1, b+j)$-$k$-omino tower.  The uniqueness follows similarly for shifts greater than one.  Thus compositions of different sizes produce unique towers.  Furthermore, two compositions of the same size must also produce different towers.  The right staircase given by $R_1, \ldots, R_j$ is unique because the sums $k_1+ \cdots + k_j$ are unique as they represent all compositions of the integers between $j$ and $k-j$ into $j$ parts.

Finally, it is left to consider the case where $b=1$ which was first studied by Durhuss and Eilers~\cite{Durhuss_Eilers2}; see the remark below.  In this case the, results in the proof thus far hold where $L_1$ is completely supported by $L_0$, that is, the leftmost column of $L_0$ intersects $L_1$.  However, we need to consider towers where $L_1$ is not fully supported by $L_0$.  
Because the number of blocks in the base is one, the recursion onto $(n-1)$-$k$-omino towers with base $b=0$ will not correctly enumerate the $k(n-1)-1\choose n-1$ towers described in the formula. We apply the following identity,
\begin{equation*}
{k(n-1)-1\choose n-1} = (k-1){k(n-1)-1\choose n}.
\end{equation*}
Thus, we need to show that set of $(n,1)$-$k$-omino towers where $L_1$ does not intersect the column containing the left end of $L_0$ is in bijection with $(k-1)$ copies of the set of $(n-1,1)$-$k$-omino towers.  This is done by placing a $(n-1,1)$-$k$-omino tower on a single $k$-omino in each of the $(k-1)$ positions so that the base of the $(n-1,1)$-$k$-omino tower hangs over the new $k$-omino base on the right.  In particular, the first two terms of the sum in Equation~\ref{equation_k} give all $(2k-1)$ ways to place a $(n-1,1)$-$k$-omino tower on a base of a single $k$-omino whereas the remaining summands count all towers with two $k$-ominoes on the first level.

Thus, the claim is proven.
\end{proof}

\begin{remark}
If $k=1$, the binomial coefficient ${n-1\choose n-b}$ counts compositions of $n$ into $b$ nonzero parts where the integers in the composition correspond to the number of unit blocks in each column of the $1$-omino tower.  In this case, the total number of $1$-omino towers of area $n$ is
\begin{equation*}
D_1(n) = \sum_{b=1}^{n} {n-1\choose n-b} = 2^{n-1}.
\end{equation*}
Further if $b=1$, these towers were enumerated by Durhuss and Eilers~\cite{Durhuss_Eilers2} using a bijection with strings of 0's and 1's.
\end{remark}

Now, Theorem~\ref{count_komino_theorem} on the number of $k$-omino towers composed of $n$ $k$-ominoes is an immediate consequence of Theorem~\ref{theorem_knb} and the definition of the Gaussian hypergeometric function
\begin{equation*}
_2F_1(a,b;c;z)=\sum_{k=1}^{\infty} \frac{(a)_k (b)_k}{(c)_k} \frac{z^k}{k!}
\end{equation*}
where $(x)_k$ denotes the rising Pochhammer symbol such that $(x)_0=1$ and 
\begin{equation*}
(x)_n=x(x+1)\cdots (x+n-1)
\end{equation*}
 for integers $n\geq 1$.  
Further, Theorem~\ref{count_komino_theorem} introduces an identity on hypergeometric functions that can be generalized to an identity with complex parameters, $\alpha, -\beta$ and the scaled parameter $c=k(\alpha+\beta)+1$ for some positive integer $k$, which, when $k=1$, is equivalent to the classical parameters, $\alpha,\beta$ and $c=\alpha-\beta+1$ of Kummer's Theorem~\cite{Kummer}.
\begin{theorem}
For $\alpha, \beta \in \mathbb{C}$ and $k\in \mathbb{N}_{+}$ where $k\alpha+k\beta+1$ is not zero or a negative integer, we have the hypergeometric identity
\begin{equation*}
{k\alpha+k\beta+\beta \choose \beta} {}_2F_1(\alpha,-\beta;k\alpha+k\beta+1; -1) = \sum_{i\geq 0} {k\alpha+k\beta+\beta \choose \beta-i} \frac{(\alpha)_i}{i!}
\end{equation*}
where $x \choose y$ denotes the extended binomial coefficient $\frac{\Gamma(x+1)}{\Gamma(y+1) \Gamma(x-y+1)}$.
\end{theorem}
\begin{proof}
The proof follows directly by multiplying the extended binomial coefficient, expanded in terms of the Gamma function, across the sum given by the hypergeometric function.
\end{proof}

\noindent \textbf{References} 


\begin{thebibliography}{99}
\bibitem{Barcucci_Frosini_Rinaldi} E. Barcucci, A. Frosini, and S. Rinaldi, {\it On directed-convex polyominoes in a rectangle}, Discrete Math. {\bf 298} (2005), 62--78. 
\bibitem{Barcucci_Lungo_Fedou_Pinzani} E. Barcucci, A. Del Lungo, J.M. F\'edou, and R. Pinzani, {\it Steep polyominoes, $q$-Motzkin numbers and $q$-Bessel functions}, Discrete Math. {\bf 189} (1998), 21--42. 
\bibitem{Barequet_Rote_Shalah} G. Barequet, G. Rote, M. Shalah, {\it $\lambda >4$: An improved lower bound on the growth constant of polyominoes}, Communications of the ACM, {\bf 59}, 7 (2016), 88-95.
\bibitem{Bender} E. Bender, {\it Convex $n$-ominoes}, Discrete Math. {\bf 8} (1974), 219--226. 
\bibitem{Betrema_Penaud} J. B\'etr\'ema and J.G. Penaud, {\it Mod\`eles avec particules dures, animaux dirig\'es, et s\'eries en variables partiellement commutatives}, (1993) , ArXiv Preprint, \url{arXiv:math.CO/0106210}.
\bibitem{Bousquet_Rechnitzer} M. Bousquet-M\'elou and A. Rechnitzer, {\it Lattice animals and heaps of dimers}, Discrete Math. {\bf 258} (2002), 235--274. 
\bibitem{Boussicault_Rinaldi_Socci}A. Boussicault, S. Rinaldi, and S. Socci, {\it The number of directed k-convex polyominoes}, 27th International Conference on Formal Power Series and Algebraic Combinatorics (FPSAC'15), 511-512, Discrete Math. Theoret. Comput. Sci. Proc., BC, Assoc. Discrete Math. Theor. Comput. Sci, Nancy, 2015.
\bibitem{Bouvel_Guerrini_Rinaldi} M. Bouvel, V. Guerrini, and S. Rinaldi, {\it Slicings of parallelogram polyominoes, or how Baxter and Schr\"oder can be reconciled},  28th International Conference on Formal Power Series and Algebraic Combinatorics (FPSAC'16), 287-298, Discrete Math. Theoret. Comput. Sci. Proc., BC, Assoc. Discrete Math. Theor. Comput. Sci, Nancy, 2016.
\bibitem{Castiglione_etal_2005} G. Castiglione, A. Frosini, A. Restivo, and S. Rinaldi, {\it Enumeration of L-convex polyominoes by rows and columns}, Theoret. Comput. Sci. {\bf 347} (2005), 336-352.
\bibitem{Castiglione_etal_2007} G. Castiglione, A. Frosini, E. Munarini, A. Restivo, and S. Rinaldi, {\it Combinatorial aspects of L-convex polyominoes}, Eur. J. Comb. {\bf 28} (2007), 1724-1741. 
\bibitem{Delest_Viennot} M.-P. Delest and G. Viennot, {\it Algebraic languages and polyominoes enumeration}, Theoret. Comput. Sci. {\bf 34} (1984), 169-206. 
\bibitem{Dhar} D. Dhar, {\it Equivalence of the two-dimensional directed-site animal problem to Baxter's hard square lattice gas model}, Phys. Rev. Lett. {\bf 49} (1982), 959-962.
\bibitem{Domocos}V. Domo\c{c}os, {\it A combinatorial method for the enumeration of column-convex polyominoes}, Discrete Math. {\bf 152} (1996), 115--123. 
\bibitem{Duchi_Rinaldi_Schaeffer} E. Duchi, S. Rinaldi, and G. Schaeffer, {\it The number of Z-convex polyominoes}, Adv. Appl. Math. {\bf 40} (2008), 54-72.
\bibitem{Durhuss_Eilers2} B. Durhuus and S. Eilers, {\it Combinatorial aspects of pyramids of one-dimensional pieces of fixed integer length}, 21st International Meeting on Probabilistic, Combinatorial, and Asymptotic Methods in the Analysis of Algorithms (AofA'10), 143-158, Discrete Math. Theoret. Comput. Sci. Proc., AM, Assoc. Discrete Math. Theor. Comput. Sci, Nancy, 2010.
\bibitem{Fedou_Frosini_Rinaldi} J. M. Fedou, A. Frosini, S. Rinaldi, {\it Enumeration of 4-stack polyominoes}, Theoret. Comput. Sci. {\bf 502} (2013) 88-97.
\bibitem{Golomb_AMM}S. Golomb, {\it Checkerboards and Polyominoes}, Amer. Math. Monthly, {\bf 61} (1954), 675--682.
\bibitem{Golomb}S. Golomb, {\it Polyominoes}, Princeton University Press; 2nd edition, Princeton, N.J. (1996). 
\bibitem{Gould} H. W. Gould, Combinatorial Identities, self-published, Morgantown WV (1972). 
\bibitem{Jensen} I. Jensen and A. Guttmann, {\it Statistics of lattice animals (polyominoes) and polygons}, J. Phys. A, {\bf 33} (2000), L257-L263. 
\bibitem{Klarner_Rivest_2} D. Klarner and R. Rivest, {\it A procedure for improving the upper bound for the number of n-ominoes}, Canad. J. Math. {\bf 25} (1973), 585-602. 
\bibitem{Klarner_Rivest} D. Klarner and R. Rivest, {\it Asymptotic Bounds for the Number of Convex $n$-ominoes}, Discrete Math. {\bf 8} (1974), 31--40. 
\bibitem{Kummer} E. E. Kummer, {\it Ueber die hypergeometrische Reihe}, J.\ f\"ur Math. {\bf 15} (1836), 39-83.
\bibitem{Viennot} G. Viennot, {\it Heaps of pieces 1: basic definitions and combinatorial lemmas},  Graph theory and its applications: East and West (Jinan, 1986), Ann. New York Acad. Sci., {\bf 576}, New York Acad. Sci., New York, (1989), 542--570.
\bibitem{Wright} E. M. Wright, {\it Stacks}, Quart. J. Math. {\bf 2} (1968), 313--320. 
\end{thebibliography}
\end{document}